\documentclass[a4paper,11pt]{amsart}

\usepackage{amssymb,mathrsfs,enumitem,amsmath,amsthm,bbold,tikz-cd}
\usepackage[mathscr]{euscript}
\usepackage{color}
\usepackage{pgf}

\usepackage[all]{xy}
\definecolor{Chocolat}{rgb}{0.36, 0.2, 0.09}
\definecolor{BleuTresFonce}{rgb}{0.215, 0.215, 0.36}
\definecolor{EgyptianBlue}{rgb}{0.06, 0.2, 0.65}

\usepackage{fourier}

\usepackage[colorlinks,final,hyperindex]{hyperref}
\usepackage{cleveref}
\hypersetup{citecolor=BleuTresFonce, urlcolor=EgyptianBlue, linkcolor=Chocolat}

\usepackage[backend=bibtex,
    sorting=anyt,
    isbn=false,
    url=false,
    doi=false,
    maxbibnames=100
    ]{biblatex}
\newtheorem{theorem}{Theorem}
\newtheorem{corollary}[theorem]{Corollary}

\newtheorem{lemma}[theorem]{Lemma}
\newtheorem{proposition}[theorem]{Proposition}

\theoremstyle{definition}

\newtheorem{remark}[theorem]{Remark}
\newtheorem{definition}[theorem]{Definition}

\newcommand{\pic}[1]{\vcenter{\hbox{\includegraphics[height=0.6cm]{NCP_#1.pdf}}}}

\DeclareMathAlphabet{\pazocal}{OMS}{zplm}{m}{n}

\DeclareMathOperator{\Yam}{\textsl{Yam}}

\DeclareMathAlphabet{\mathbbold}{U}{bbold}{m}{n}

\def\k{\mathbbold{k}}

\def\lb{\{ \! \! \{}
\def\rb{\} \! \! \}}

\title{Yamaguti algebras and noncrossing partitions}

\author{Frédéric Chapoton}
\address{Institut de Recherche Math\'ematique Avanc\'ee, UMR 7501, Universit\'e de Strasbourg et CNRS, 7 rue Ren\'e-Descartes, 67000 Strasbourg, France}
\email{chapoton@unistra.fr}

\author{Vladimir Dotsenko}
\address{Institut de Recherche Math\'ematique Avanc\'ee, UMR 7501, Universit\'e de Strasbourg et CNRS, 7 rue Ren\'e-Descartes, 67000 Strasbourg, France}
\email{vdotsenko@unistra.fr}

\date{\today}

\thanks{This work was supported by the ANR project HighAGT (ANR-20-CE40-0016).}

\begin{document}

\begin{abstract}
Recently, Das defined a new type of algebras, the Yamaguti algebras, which are supposed to serve as envelopes of Lie--Yamaguti algebras appearing naturally in differential geometry. We show that the nonsymmetric operad of Yamaguti algebras admits a combinatorial description via noncrossing partitions without singleton blocks, and a representation-theoretic description as the equivariant endomorphism operad of the adjoint module of the Lie algebra $\mathfrak{sl}_2$. 
\end{abstract}

\maketitle

\section{Yamaguti algebras}

In work of Nomizu~\cite{MR59050} on affine connections with parallel torsion and curvature, a new algebraic structure with one binary and one ternary structure operation emerged. It was axiomatized algebraically by Yamaguti~\cite{MR100047} as ``general Lie triple systems'', and renamed into ``Lie triple algebras'' by Kikkawa~\cite{MR383301}. Much later, it was renamed into ``Lie--Yamaguti algebras'' by Kinyon and Weinstein \cite{MR1833152}, which seems to be the preferred terminology these days. While the structure theory of Lie--Yamaguti algebras was studied in recent years~\cite{MR2494372,MR2720677}, it is not a very well understood algebraic structure. For instance, a basis in the free Lie--Yamaguti algebra does not seem to be known (a recent preprint of Stava~\cite{stava} that proposes such a basis implicitly proves that the subalgebra of the free Lie--Yamaguti algebra obtained by iterated binary products of generators is free, an assertion shown to be false by Bremner~\cite{MR3265628}). This is one of the motivations of the recent work of Das~\cite{das}, who defined a new type of algebras, which he called associative--Yamaguti algebras, or simply Yamaguti algebras. For these algebras, suitable symmetrization of their operations produces a Lie--Yamaguti algebra, and hence one may hope to study Lie--Yamaguti algebras via their Yamaguti envelopes. The precise definition of Yamaguti algebras is as follows.

\begin{definition}
A \emph{Yamaguti algebra} is a vector space equipped with a bilinear operation $a,b\mapsto a\cdot b$ and two trilinear operations 
 \[
 a,b,c\mapsto \{a,b,c\}\quad \text{and} \quad a,b,c\mapsto\lb a,b,c\rb
 \]
satisfying the identities
   \begin{equation}\label{ay1} \tag{AY1}
        (a \cdot b ) \cdot c -  a \cdot (b \cdot c) + \{ a, b, c \} - \lb a, b, c \rb = 0,
    \end{equation}
    \begin{equation}\label{ay2} \tag{AY2}
        \{ a \cdot b, c, d \} = \{ a, b \cdot c, d \},
    \end{equation}
    \begin{equation}\label{ay3} \tag{AY3}
        \{ a, b, c \cdot d \} = \{ a, b, c \} \cdot d,
    \end{equation}
    \begin{equation}\label{ay4} \tag{AY4}
        \lb a \cdot b, c, d\rb = a \cdot \lb b, c, d \rb,
    \end{equation}
    \begin{equation}\label{ay5} \tag{AY5}
        \lb a, b \cdot c, d\rb = \lb a, b, c \cdot d \rb ,
    \end{equation}
    \begin{equation}\label{ay6} \tag{AY6}
        a \cdot \{ b, c, d \} = \lb a, b, c \rb \cdot d,
    \end{equation}
    \begin{equation}\label{ay7} \tag{AY7}
         \{ \{ a, b, c \}, d, e \} = \{ a, \lb b, c, d\rb, e \} = \{ a, b, \{ c, d , e \} \},
    \end{equation}
    \begin{equation}\label{ay8} \tag{AY8}
        \{ a, \{ b, c , d \} , e \} = \{ \lb a, b, c\rb, d, e \},
    \end{equation}
    \begin{equation}\label{ay9} \tag{AY9}
        \lb   \lb a, b, c\rb, d, e  \rb = \lb a, \{ b, c, d \}, e \rb = \lb a, b, \lb c, d, e\rb \rb,
    \end{equation}
    \begin{equation}\label{ay10} \tag{AY10}
         \lb a, \lb b, c, d \rb, e \rb  = \lb a, b, \{c , d, e \} \rb,
    \end{equation}
    \begin{equation}\label{ay11}\tag{AY11}
        \{ a, b, \lb c, d, e \rb \} = \lb \{ a, b, c \}, d, e \rb.
    \end{equation}
\end{definition} 

Note that in each Relation~\eqref{ay1}--\eqref{ay11} the arguments appear in the same order, which means that a Yamaguti algebra is an algebra over a nonsymmetric operad, which we denote $\Yam$, and call the Yamaguti operad. In this note, we show that this operad has two other descriptions: a very simple combinatorial one, and one relating it to representation theory of $\mathfrak{sl}_2$; one may hope that these results, besides being interesting in their own right, might also be useful to understand the more elusive Lie--Yamaguti operad. It also follows from our description that the operad $\Yam$ is cyclic; the same statement for the Lie--Yamaguti operad is implicit in the literature, since there exists a meaningful notion of an invariant bilinear form on a Lie--Yamaguti algebra \cite{MR4442492}.

All vector spaces in this paper are defined over a field $\k$; for some of the results, it will be important to assume $\mathrm{char}(\k)\ne 2$.

\section{Operad of noncrossing partitions without singleton blocks}

Let $B(n)$ be the set of noncrossing partitions of the set
$\{0,1,2,\ldots,n\}$ without singleton blocks; these combinatorial objects were first considered by Kreweras \cite[Sec.~5]{MR309747}. We display such a noncrossing partition by a picture where integers are segments on the
boundary of a disk, with the segment $0$ at the bottom. The blocks
are then represented by blue regions inside the disk as in
Figure~\ref{fig1}.

\begin{figure}[h]
  \centering
  \includegraphics[height=3cm]{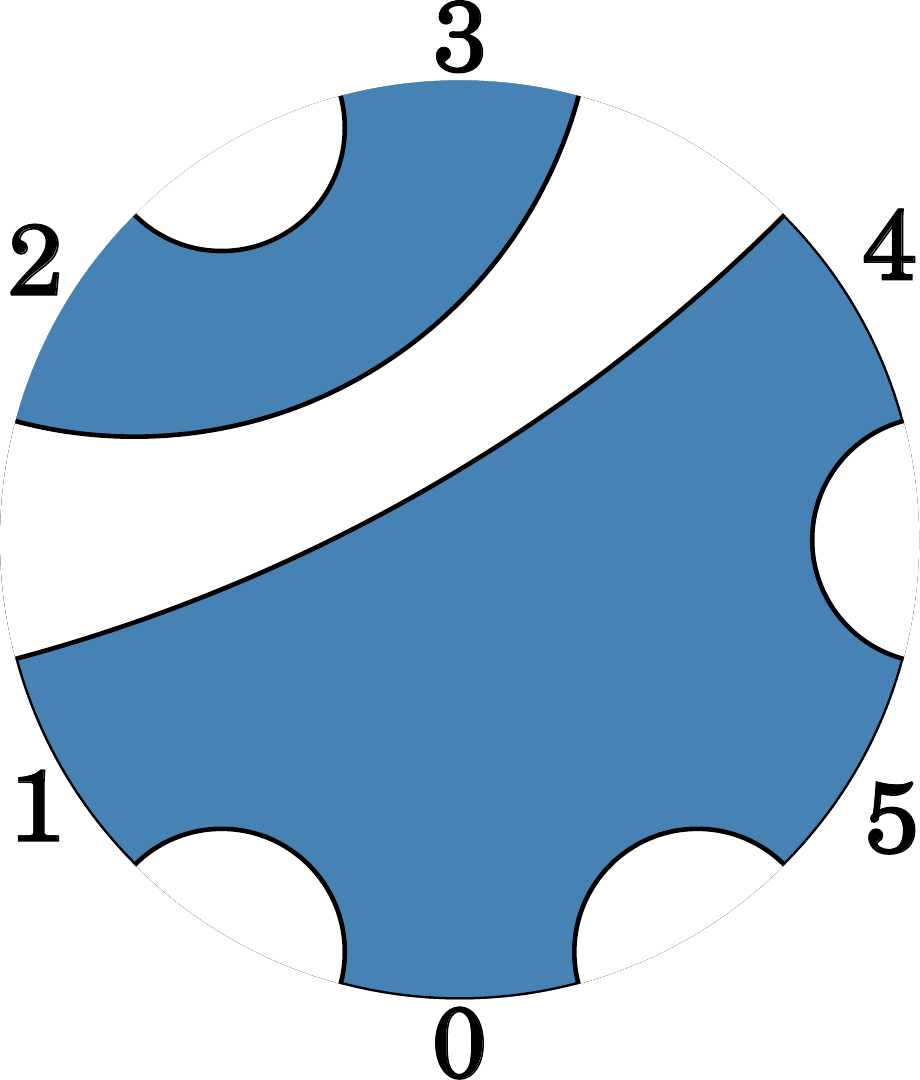}
  \caption{The noncrossing partition with blocks $\{0,1,4,5\}$ and $\{2,3\}$.}\label{fig1}
\end{figure}

Let us define a nonsymmetric operad $\mathscr{B}$ in the category of vector
spaces. The component $\mathscr{B}(n)$ of arity $n$ is the vector space with basis $B(n)$. 

The composition maps $\circ_i$ are defined on the basis as
follows. Let $\pi \in B(m)$, $\nu \in B(n)$ and $1 \leq i \leq
m$. Then $\pi \circ_i \nu$ is the sum of two terms:
\begin{itemize}
\item the noncrossing partition obtained by juxtaposition of the two
  disks and identification of the segment $i$ of $\pi$ with the
  segment $0$ of $\nu$,
\item the noncrossing partition obtained from the previous one by
  cutting the new block along the gluing line made from the
  identified segments, thereby creating two blocks.
\end{itemize}
It can happen that the second case creates a singleton block, in which
case this term is omitted. This happens exactly when either the block in
$\nu$ containing $0$ has $2$ elements or the block in $\pi$ containing
$i$ has $2$ elements.

For example, one gets
\begin{equation*}
  \pic{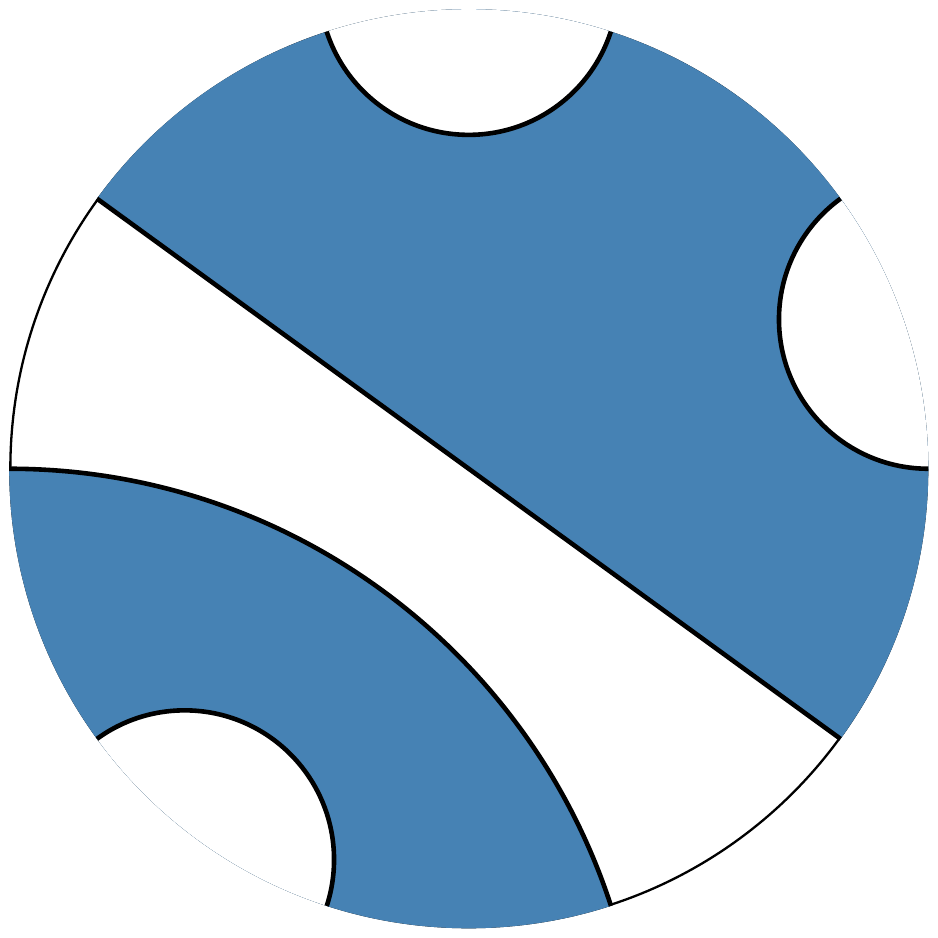} \circ_3 \pic{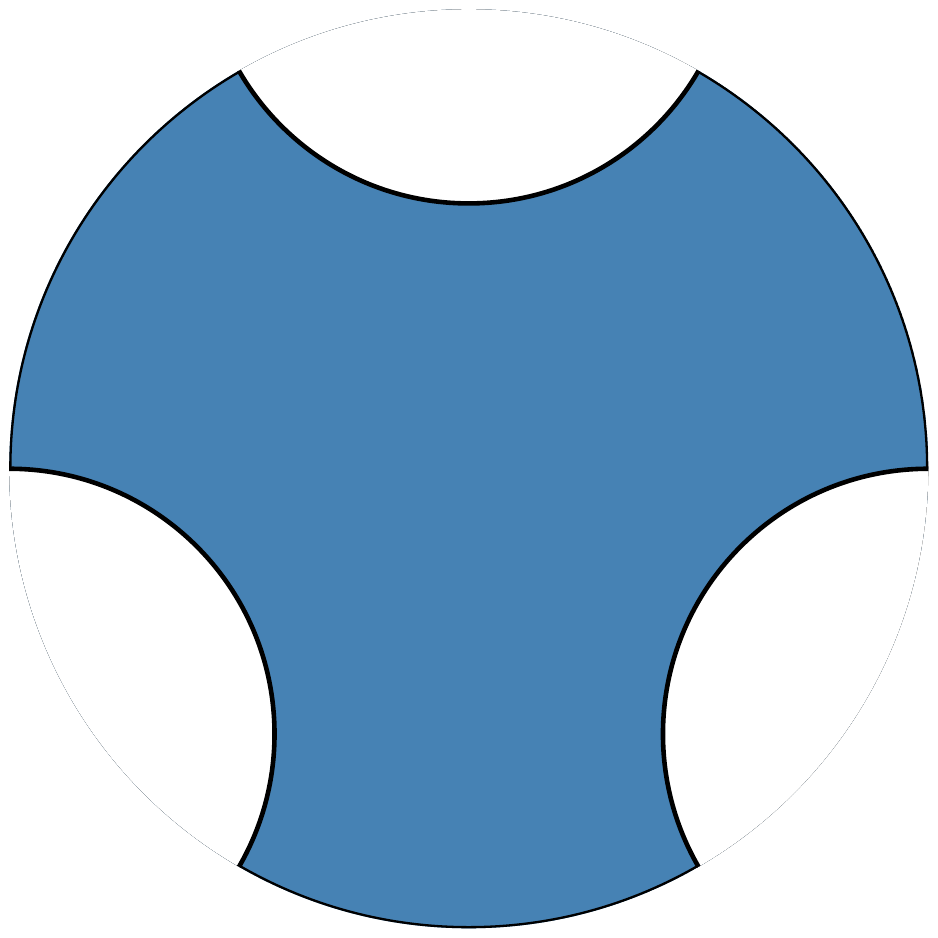} = \pic{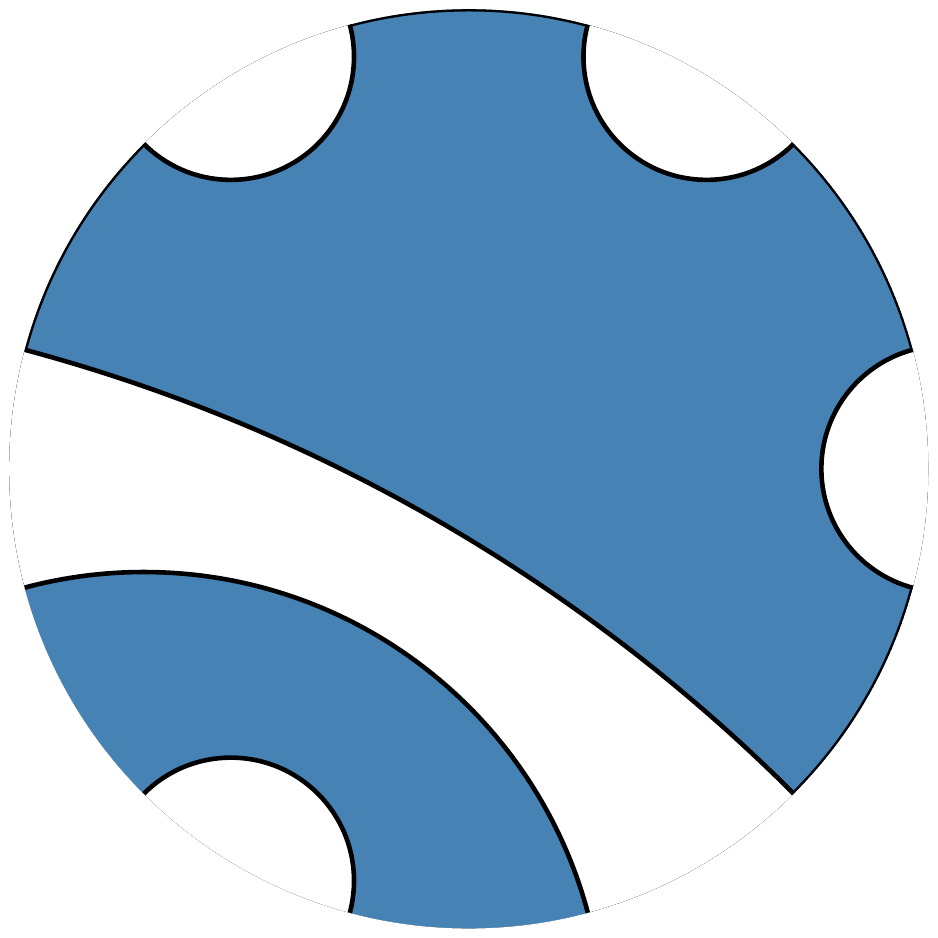} + \pic{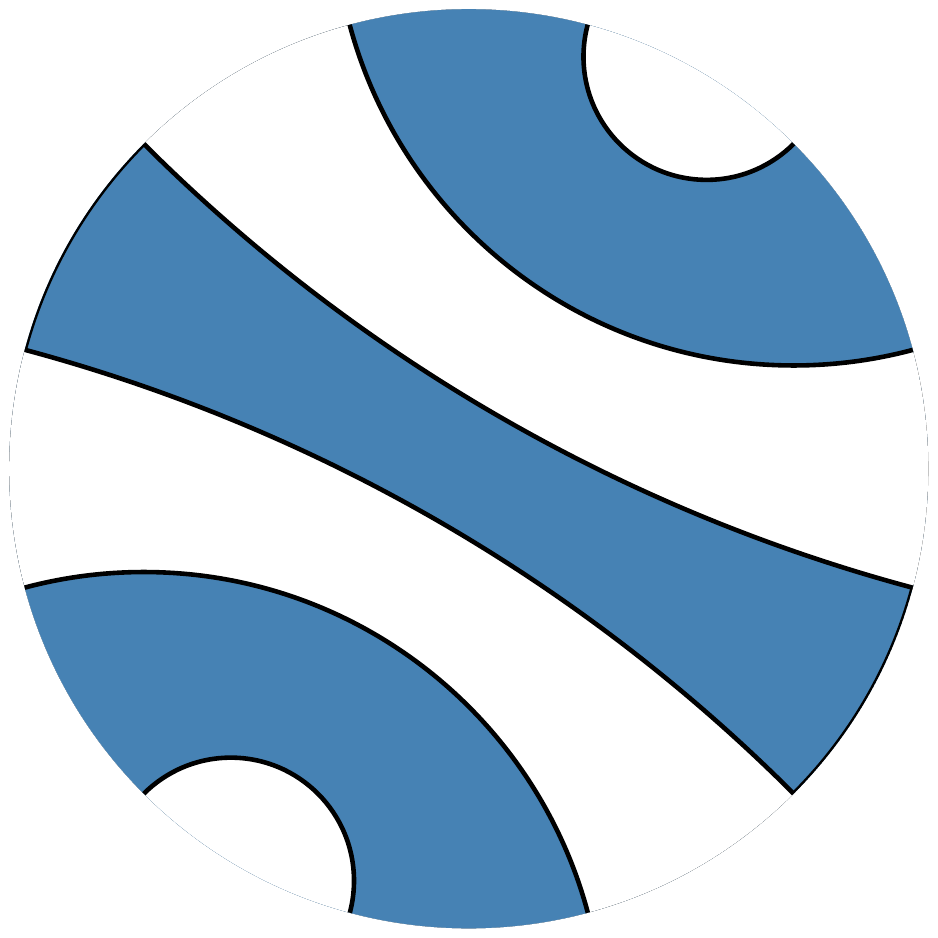}.
\end{equation*}

\begin{proposition}
The thus defined operations $\circ_i$ make $\mathscr{B}$ into a nonsymmetric cyclic operad.
\end{proposition}

\begin{proof}
First of all, we note that the element $\pic{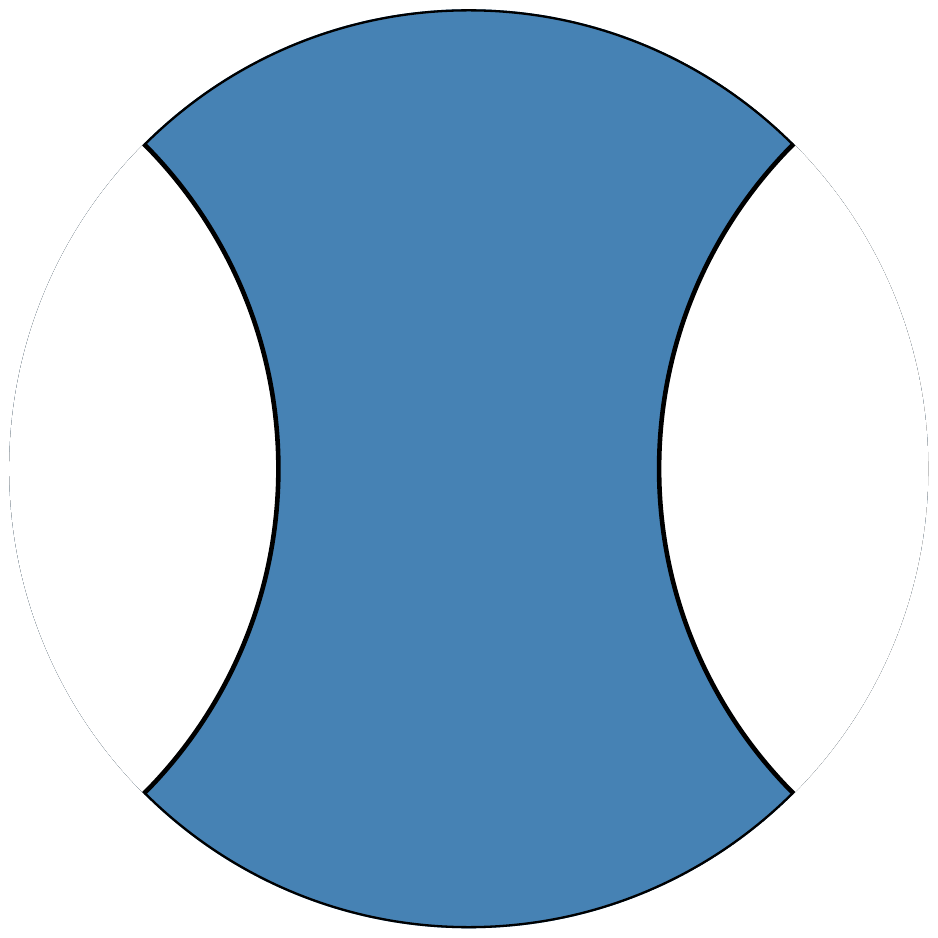}$ satisfies the axioms of the operadic unit, because when
composing with this element, only the first term appears. Next, we note that our definition of the compositions is immediately seen compatible with the cyclic group action on $B(n)$, since the combinatorics of these compositions only uses the planar depiction of noncrossing partitions.

It remains to check the two other axioms of nonsymmetric operads, that is, the axioms of the parallel and the sequential composition \cite[Chapter 5]{LV}.
For the parallel axiom $(\pi \circ_i \mu) \circ_j \nu=(\pi \circ_j \nu) \circ_i \mu$ corresponding to the simultaneous composition at positions $i$ and $j$ inside a noncrossing partition $\pi$, the statement is obvious if $i$ and $j$ are not in the same block of $\pi$, as the composition rules play independently at $i$ and $j$. One has to examine what happens when $i$
and $j$ are in a block $b$ of $\pi$. If $|b|$ is at least $4$, the rules also play independently. If $|b|$ is $2$ or $3$, one can check
all possible cases. For instance, if $b=\{i,j\}$, then 
\begin{itemize}
\item[-] in each of the two compositions $\pi \circ_i \mu$ and $\pi \circ_j \nu$ only the first term appears, 
\item[-] each of the two compositions $(\pi \circ_i \mu) \circ_j \nu$ and $(\pi \circ_j \nu) \circ_i \mu$ has two terms, one corresponding to identification of the segments $i$ and $j$ of $\pi$ with the segments $0$ of $\mu$ and of $\nu$ respectively, and the other  
obtained from the previous one by cutting the new block into its $\mu$-part and its $\nu$-part, thereby creating two blocks. (The second term is not present if either of the blocks in $\mu$ and $\nu$ containing $0$ has $2$ elements.)
\end{itemize}
It follows that the sequential composition axiom also holds, since it can be reduced to the parallel axiom by a cyclic group action, which completes the proof. 
\end{proof}

\begin{remark}
  We note that a nonsymmetric operad of noncrossing partitions
  (without restrictions on blocks) was defined by Ebrahimi--Fard,
  Foissy, Kock and Patras~\cite{MR4093606}, who used it in the context
  of moment-cumulant relations in free probability. However, their
  operad, unlike ours, is set-theoretic, and uses partitions of
  $\{1,\ldots,n\}$ as operations of arity $n+1$, so there does not
  seem to be any obvious relationship between the two.
\end{remark}

\begin{lemma}\label{lm:outer_block}
  Let $\pi$ be a noncrossing partition of $\{0,1,\ldots,n\}$ without
  singleton blocks and with at least $2$ blocks, for some $n \geq
  3$. Then $\pi$ has a block made of a sequence of consecutive
  non-zero integers.
\end{lemma}
\begin{proof}
  The unique block $b$ containing $0$ defines a partition of the other
  blocks according to which connected component of the complement of
  $b$ they belong. Because there are at least $2$ blocks, at least one
  of these connected components contains a block. Iterating this
  process by choosing blocks further away from $b$, one must reach at
  some step a block $b'$ whose elements are a sequence of consecutive
  non-zero integers.
\end{proof}

\begin{proposition}\label{prop:gen_B}
  The operad $\mathscr{B}$ is generated by the elements
  \begin{equation*}
    \pic{2} \in \mathscr{B}(2)\quad\text{ and } \quad \pic{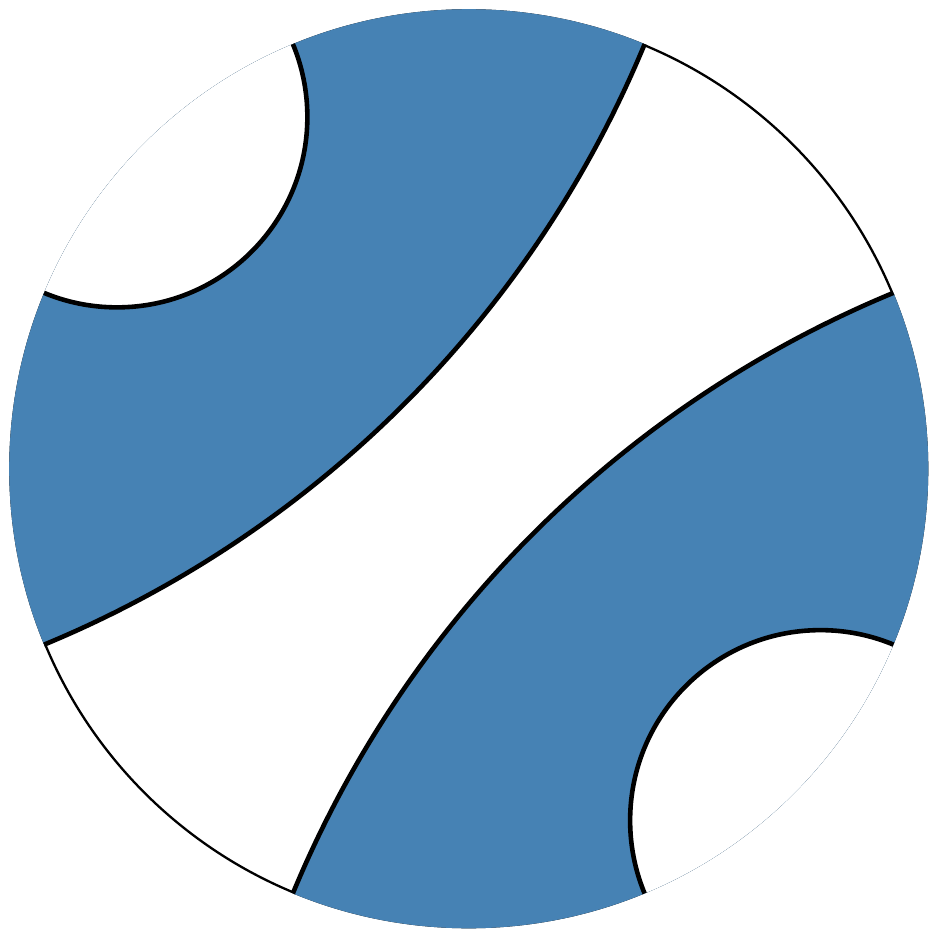}\,\, ,\,\, \pic{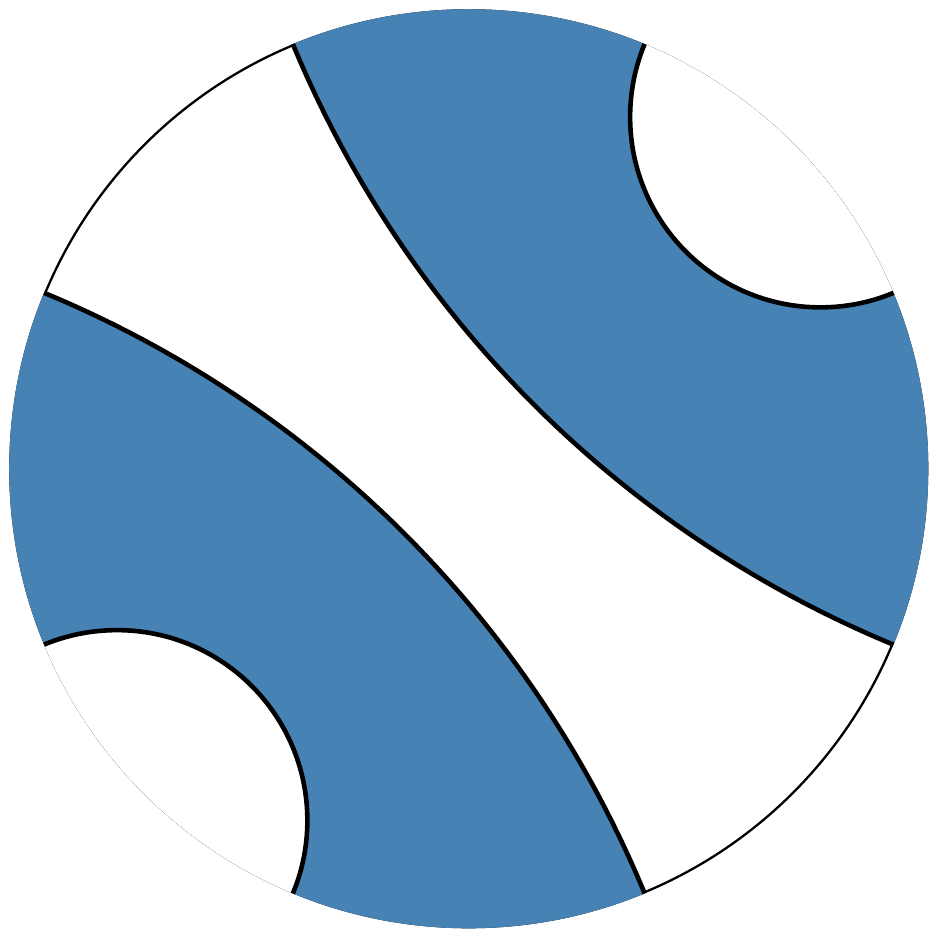}\in \mathscr{B}(3).
  \end{equation*}
\end{proposition}
\begin{proof}
  Let $\mathscr{G}$ be the sub-operad generated by the elements indicated above. Let us prove by induction on the arity that $\mathscr{G}$ contains every basis element. 

  The statement is clear in arity at most $2$. Let $\pi$ be a basis element of arity $n \geq 3$.

  Assume first that $\pi$ has at least $2$ blocks. By Lemma~\ref{lm:outer_block}, $\pi$ contains a block $b$ made of
  consecutive non-zero integers $\{i,i+1,\ldots,j\}$.

  If the block $b$ has at least $3$ elements, let $\nu$ be defined by
  replacing $b$ by a block with two elements $\{i,i+1\}$ and
  renumbering the integers after $j+1$. Let $\mu$ be the noncrossing
  partition with one part of cardinality $|b|$. Then
  $\pi = \nu \circ_{i+1} \mu$. By induction, $\pi$ is on $\mathscr{G}$.

  If the block $b$ has $2$ elements, let $\nu$ be defined by removing
  the block $b$ and renumbering the integers after $i+2$. Then $\pi$
  can be written as some composition $\nu \circ_k \pic{3_c}$ or
  $\nu \circ_k \pic{3_b}$ for some appropriate choice of
  $k$. Therefore $\pi$ is in $\mathscr{G}$.

  There remains to handle the case where $\pi$ has just one block
  $b$. The composition of two smaller noncrossing partitions with just
  one block, both in arity at least $2$, gives $\pi$ plus a
  noncrossing partition $\pi'$ with two blocks. We already know that
  $\pi'$ is in $\mathscr{G}$, so that $\pi$ is also in $\mathscr{G}$.
\end{proof}

\section{Yamaguti operad and noncrossing partitions}

We shall now establish the first main result of this paper, showing that the combinatorially defined operad $\mathscr{B}$ and the Yamaguti operad $\Yam$ are isomorphic. The isomorphism will be implemented by the map constructed as follows.

\begin{proposition}
There is a well-defined morphism
 \[
\psi\colon\Yam\to \mathscr{B}   
 \]
defined on the generators by the formula
\begin{align}
  -\cdot- &\longmapsto \pic{2},\\
  \{ -, -, - \}    & \longmapsto - \pic{3_c},\\
  \lb -, -, - \rb    & \longmapsto - \pic{3_b}.
\end{align}
\end{proposition}

\begin{proof}
To prove the above assertion, one needs to check that all relations~\eqref{ay1}--\eqref{ay11} hold for the images of the generators of $\Yam$.
One finds that
\begin{equation*}
  \pic{2} \circ_1 \pic{2} = \pic{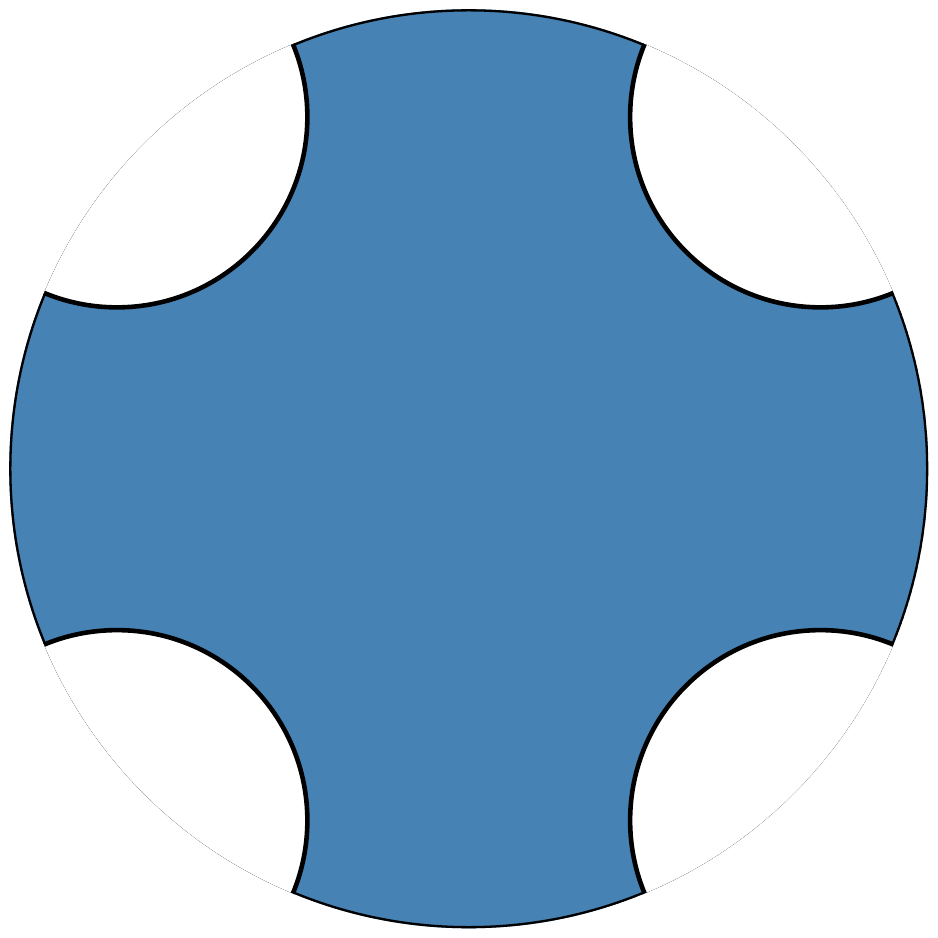} + \pic{3_c},\quad\text{and}\quad
  \pic{2} \circ_2 \pic{2} = \pic{3_a} + \pic{3_b},
\end{equation*}
so that 
 \[
\pic{2} \circ_1 \pic{2} - \pic{2} \circ_2 \pic{2} = \pic{3_c} - \pic{3_b},    
 \]
and therefore Relation~\eqref{ay1} holds. For all the other relations, the computation is in fact a bit simpler, since they all involve the ternary generators 
 \[
\pic{3_c} \quad \text{ and  } \quad \pic{3_b},   
 \]
and composition of a basis element with either of them is a single basis element. Concretely, Relations~\eqref{ay2}--\eqref{ay6} hold because 
\begin{align*}
  \pic{3_c} \circ_1 \pic{2} &= \pic{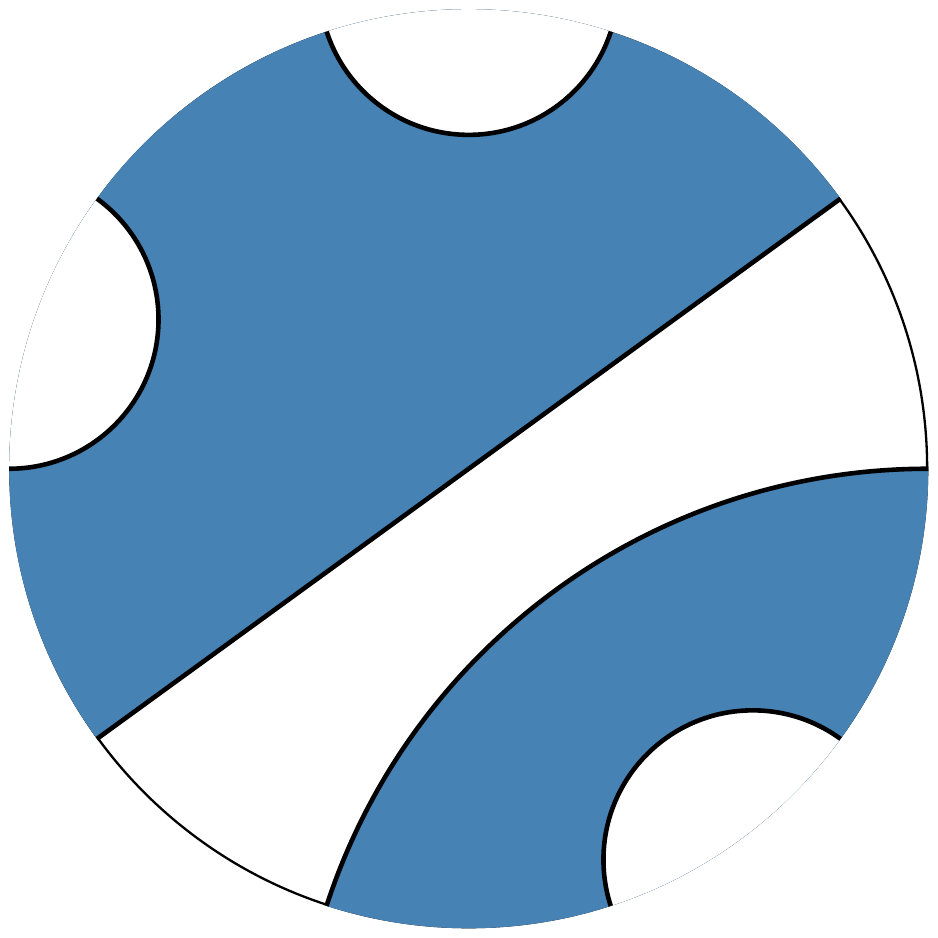} = \pic{3_c} \circ_2 \pic{2},\\
  \pic{2} \circ_1 \pic{3_c} &= \pic{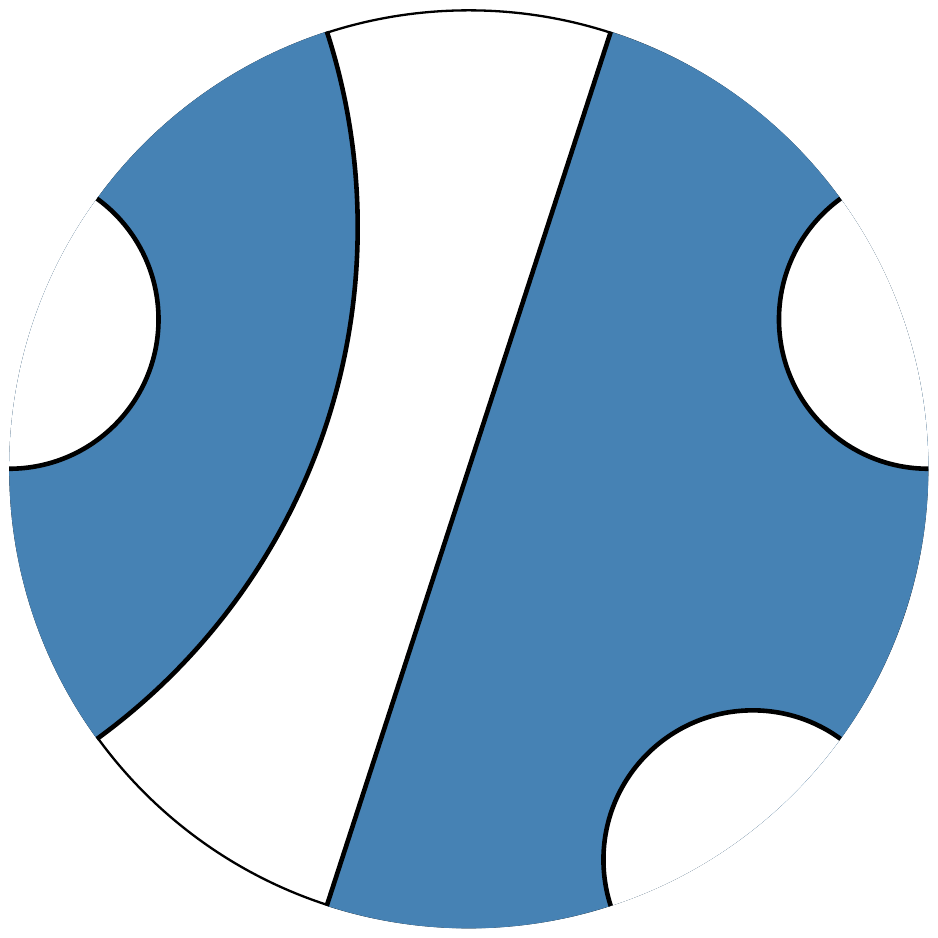} = \pic{3_c} \circ_3 \pic{2},\\
  \pic{2} \circ_2 \pic{3_b} &= \pic{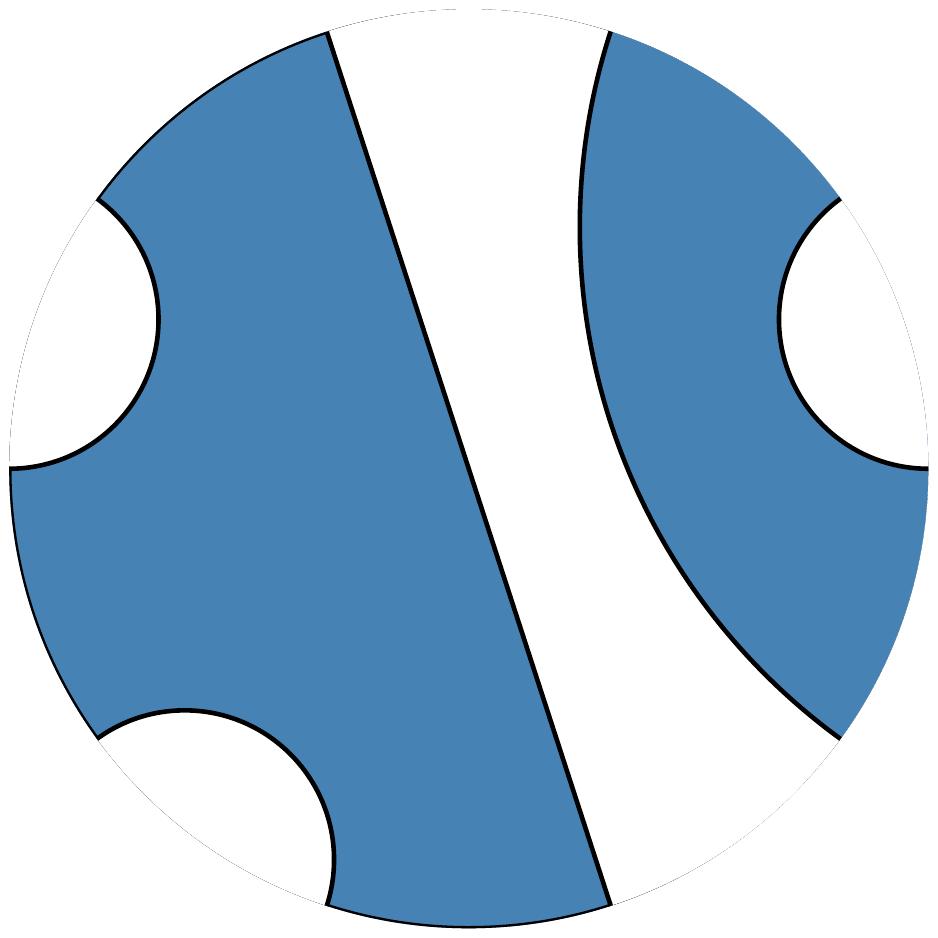} = \pic{3_b} \circ_1 \pic{2},\\
  \pic{3_b} \circ_2 \pic{2} &= \pic{4_e} = \pic{3_b} \circ_3 \pic{2},\\
  \pic{2} \circ_1 \pic{3_b} &= \pic{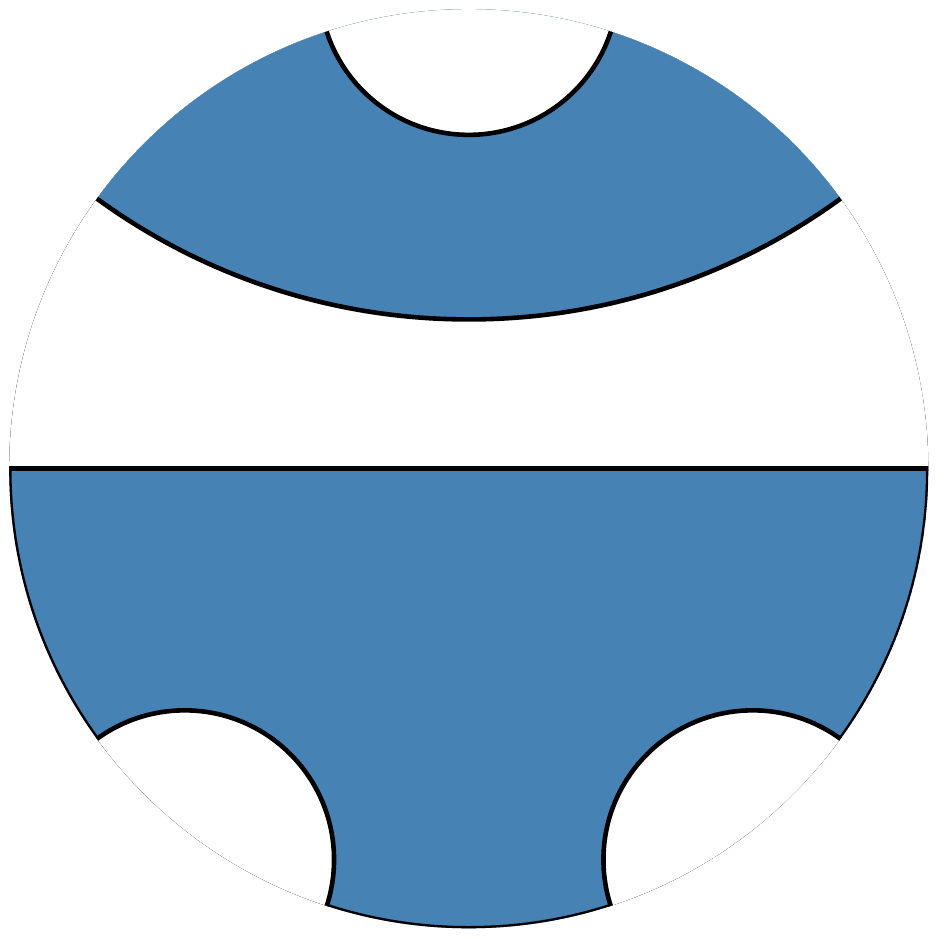} = \pic{2} \circ_2 \pic{3_c},
\end{align*}
and Relations~\eqref{ay7}--\eqref{ay11} hold because 
\begin{align*}
  \pic{3_c} \circ_1 \pic{3_c} &= \pic{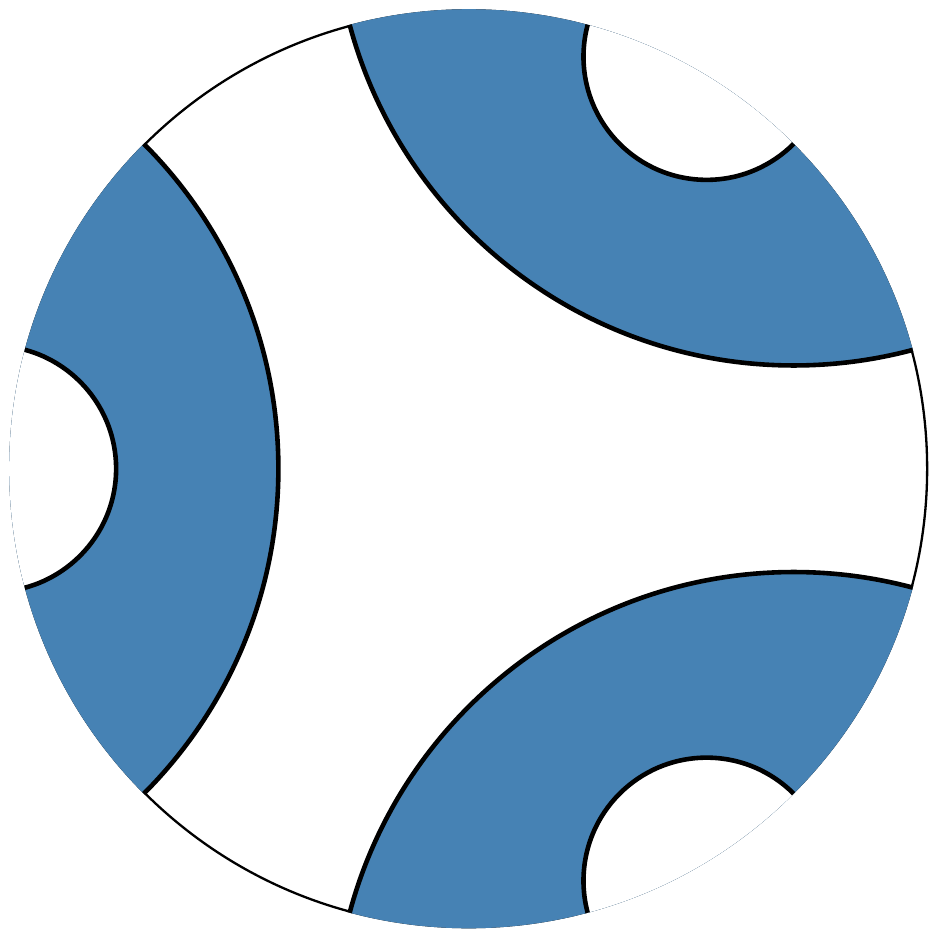} = \pic{3_c} \circ_2 \pic{3_b} = \pic{3_c} \circ_3 \pic{3_c},\\
  \pic{3_c} \circ_2 \pic{3_c} &= \pic{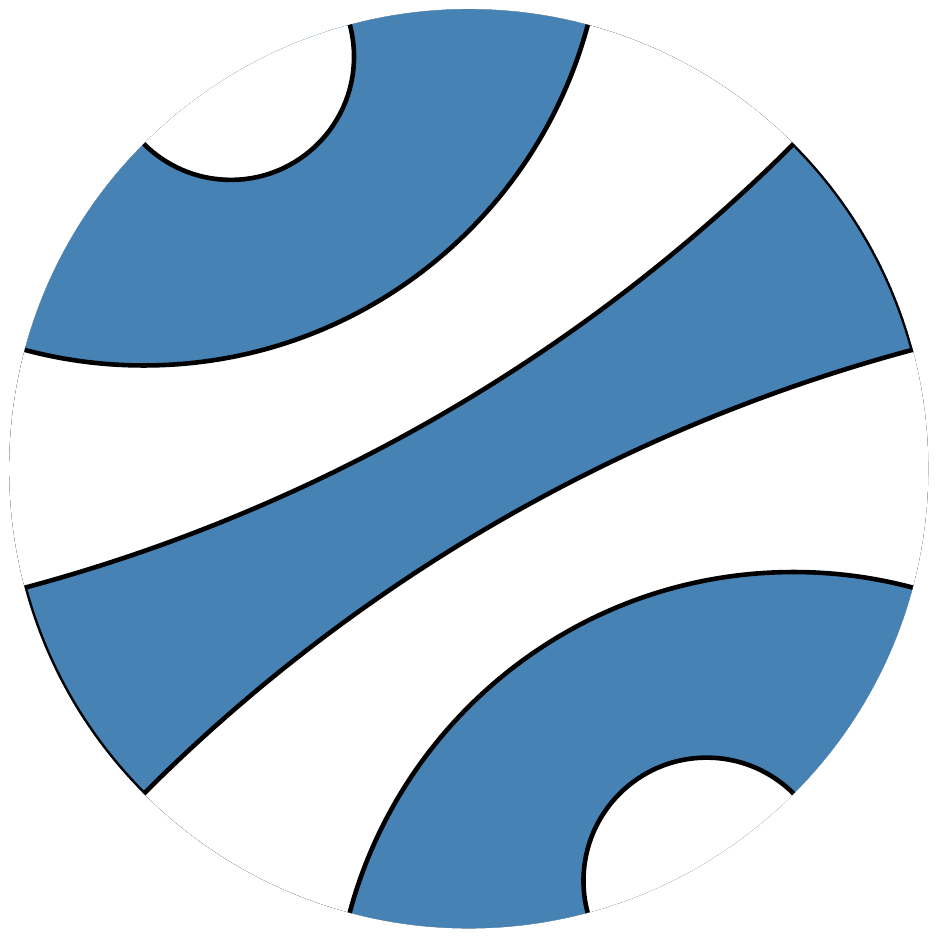} = \pic{3_c} \circ_1 \pic{3_b},\\
  \pic{3_b} \circ_1 \pic{3_b} &= \pic{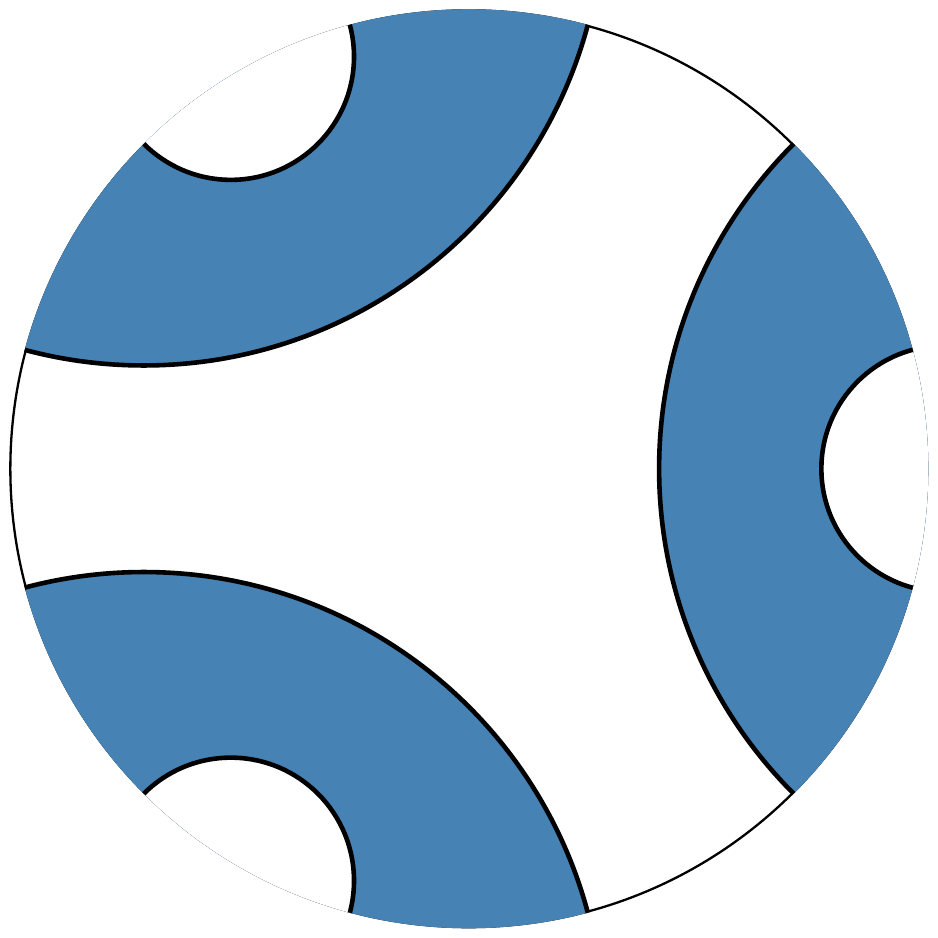} = \pic{3_b} \circ_2 \pic{3_c} = \pic{3_b} \circ_3 \pic{3_b},\\
  \pic{3_b} \circ_2 \pic{3_b} &= \pic{5_c} = \pic{3_b} \circ_3 \pic{3_c},\\
  \pic{3_c} \circ_3 \pic{3_b} &= \pic{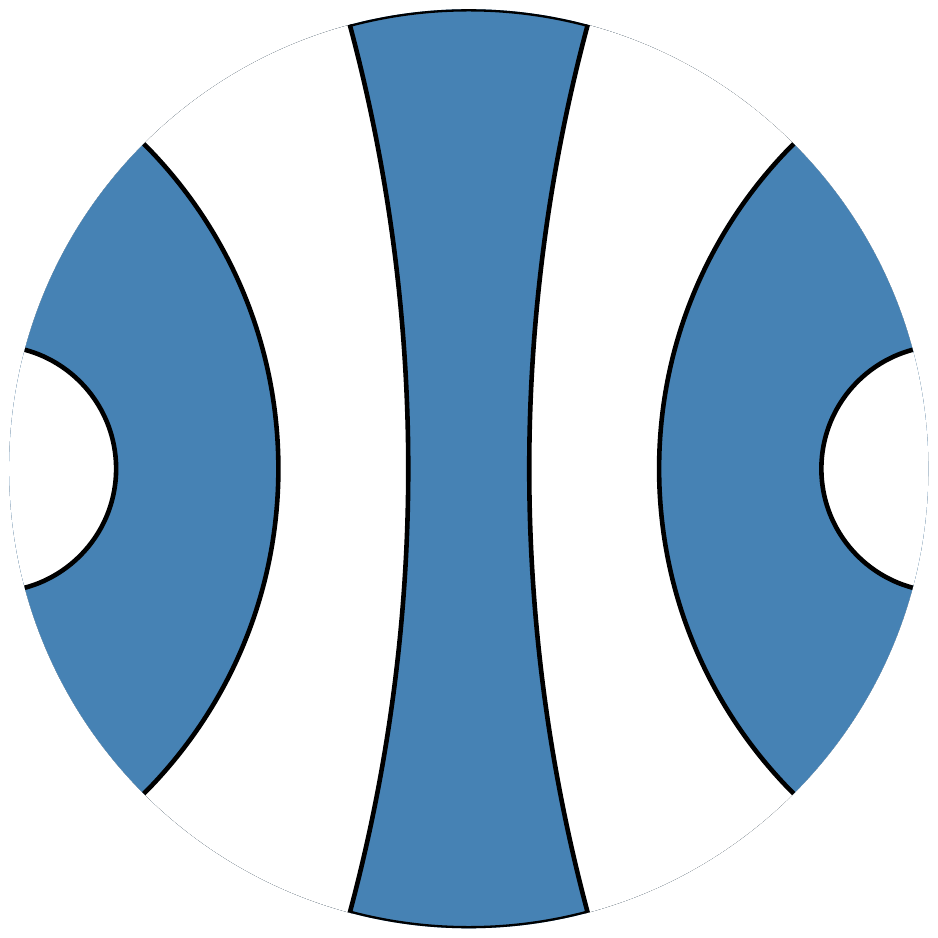} = \pic{3_b} \circ_1 \pic{3_c}.
\end{align*}
Note that the signs in the definition of $\psi$ only play a role in the first relation.
\end{proof}

\begin{theorem}\label{isom}
The morphism $\psi\colon\Yam\to \mathscr{B}$ is an isomorphism.
\end{theorem}
\begin{proof}
By Proposition~\ref{prop:gen_B}, the morphism $\psi$ is surjective, so the dimension of $\Yam(n)$ is bounded from below by the number of noncrossing partitions of $\{0,1,\ldots,n\}$ without singleton blocks. Let us show that the same number serves as an upper bound for $\dim\Yam(n)$; this would imply that $\psi$ is bijective.

We begin by noting that Identity~\eqref{ay1} can be used to write 
 \[
\lb a, b, c \rb = (a \cdot b ) \cdot c -  a \cdot (b \cdot c) + \{ a, b, c \},  
 \]
so the generator $\lb-,-,-\rb$ is redundant. Using the existing software to compute Gr\"obner bases for operads \cite{OpGb}, or eliminating the generator $\lb a, b, c \rb$ directly, one finds the following minimal set of relations between the minimal set of generators $-\cdot-$  and $\{-,-,-\}$: 
\begin{equation}\label{ay1'} \tag{AY1'}
\{ a, b \cdot c, d \} = \{ a \cdot b, c, d \},
\end{equation}
\begin{equation}\label{ay2'} \tag{AY2'}
\{ a, b, c \cdot d \} = \{ a, b, c \} \cdot d,
\end{equation}
\begin{equation}\label{ay3'} \tag{AY3'}
\{ a, b, \{ c, d , e \} \} = \{ \{ a, b, c \}, d, e \},
\end{equation}
\begin{equation}\label{ay4'} \tag{AY4'}
(a\cdot (b\cdot c))\cdot d = ((a\cdot  b) \cdot c)\cdot  d  -  a\cdot \{b, c, d\}  +  \{a, b, c\}\cdot d,
\end{equation}
\begin{equation}\label{ay5'} \tag{AY5'}
\{a\cdot (b\cdot c), d, e\} =\{(a\cdot b)\cdot c, d, e \}  - \{a, \{b, c, d\}, e\}  +  \{\{a, b, c\}, d, e\},
\end{equation}
\begin{equation}\label{ay6'} \tag{AY6'}
a\cdot (b\cdot (c\cdot d)) = a\cdot ((b\cdot c)\cdot d)  + (a\cdot b)\cdot(c\cdot d)  -  ((a\cdot b)\cdot c)\cdot d  +  a\cdot\{b, c, d\}  -  \{a\cdot b, c, d\}.
\end{equation}

Consider the ordering of monomials of the free nonsymmetric operad generated by $-\cdot-$  and $\{-,-,-\}$ which first compares the number of operations used in a monomial, and in case of a tie, compares the path sequences using the reverse graded lexicographic ordering~\cite{BD}, using the ordering of generators for which $-\cdot-$ is less than $\{-,-,-\}$. Then the leading terms of the identities given by the differences of the left hand sides and the right hand sides of~\eqref{ay1'}--\eqref{ay6'} are the monomials on the left hand sides. Thus, an upper bound for the dimension of $\Yam(n)$ is given by the number of monomials that are not divisible by any of the monomials 
 \[
\{ a, b \cdot c, d \}, \{ a, b, c \cdot d \},  \{ a, b, \{ c, d , e \} \}, (a\cdot (b\cdot c))\cdot d, \{a\cdot (b\cdot c), d, e\}, a\cdot (b\cdot (c\cdot d)).   
 \]
For a nonsymmetric operad with finitely many monomial relations, there are several different ways to compute the dimensions of its components, see, e.g.,~\cite{MR3301915}. In our case, the concrete form of the monomial relations suggests that we should introduce the following generating functions:
\begin{itemize}
\item $f(t)$ the generating function for all monomials that are not divisible by the relations, 
\item $x(t)$ the generating function for all monomials that are not divisible by the relations and $-\cdot-$ as the top level operation, 
\item $y(t)$ the generating function for all monomials that are not divisible by the relations and $\{-,-,-\}$ as the top level operation,
\item $z(t)$ the generating function for all monomials that are not divisible by the relations and $-\cdot(-\cdot-)$ at the top level.
\end{itemize} 
It is then easy to see that we have the following relations between these generating functions:
\begin{gather*}
f(t)=t+x(t)+y(t),\\
x(t)={(f(t)-z(t))}^2,\\
y(t)=(f(t)-z(t))(t+y(t))t,\\
z(t)=(f(t)-z(t))(x(t)-z(t)).
\end{gather*}
Indeed, 
\begin{itemize}
\item the first of these equations just means that a monomial is either the unit of the operad or has one of the generators at the top level,
\item the second equation means that as long as the top level operation is $-\cdot-$, we cannot have a monomial with the top level $-\cdot(-\cdot-)$ substituted as either of the two arguments, 
\item the third equation means that as long as the top level operation is $\{-,-,-\}$, we cannot have a monomial with the top level $-\cdot(-\cdot-)$ substituted as the first argument, we cannot have a monomial with the top level $-\cdot-$ substituted as the second argument, and we can only have the unit as the third argument,
\item the last equation means that to build an operation with $-\cdot(-\cdot-)$ at the top level, we must form an operation $m_1\cdot m_2$, where $m_1$ is a monomial that does not have $-\cdot(-\cdot-)$ at the top level (and thus is enumerated by $f(t)-z(t)$) and $m_2$ is a monomial that has $-\cdot-$ at the top level, but cannot have $-\cdot(-\cdot-)$ at the top level (and thus is enumerated by $x(t)-z(t)$).
\end{itemize}
Eliminating $x(t)$, $y(t)$, and $z(t)$ from these equations, we find
 \[ 
t^3{f(t)}^2 + t^2{f(t)}^2 + 2t^2f(t) + tf(t) + t - f=0,   
 \]
which, if we denote $A(t)=1+tf(t)$, can be written as 
 \[
A(t)=\frac{1}{1+t}+t{A(t)}^2,   
 \] 
the known functional equation for the generating function of the so-called Riordan numbers~\cite[A005043]{oeis}, counting noncrossing partitions of $\{0,\ldots,n\}$ without singleton blocks. This shows that the upper bound and the lower bound for $\dim\Yam(n)$ coincide, and hence the morphism $\psi$ must be an isomorphism.
\end{proof}

\begin{corollary}
We have 
 \[
\dim\Yam(n)=\sum_{k=0}^{n+1} (-1)^{n+1-k}\binom{n+1}{k}\frac{1}{k+1}\binom{2k}{k}.     
 \]
\end{corollary}

\begin{proof}
According to the proof of Theorem \ref{isom}, the generating function $A(t)$ of the Riordan numbers~\cite[A005043]{oeis} satisfies $A(t)=1+tf(t)$, so the requisite formula is obtained from the known formula for the Riordan numbers by a shift of argument.
\end{proof}

\begin{corollary}
Relations \eqref{ay1'}--\eqref{ay6'} form a Gröbner basis of relations of the operad $\Yam$.
\end{corollary}

\begin{proof}
According to the proof of Theorem \ref{isom}, the upper bound for $\dim\Yam(n)$ obtained via the leading terms of Relations~\eqref{ay1'}--\eqref{ay6'} coincides with the lower bound. Therefore, those leading terms generate the ideal of the leading terms of the operad $\Yam$, which implies that~ \eqref{ay1'}--\eqref{ay6'} form a Gröbner basis of relations of that operad.
\end{proof}

Examining the proof of Theorem \ref{isom} further, one can note that the auxiliary generating function $y(t)$ is just $t^2 f(t)$, that the coefficients of $x(t)$ are the generalized ballot numbers~\cite[A002026]{oeis} and that those of $f(t) - z(t)$
are the Motzkin numbers~\cite[A001006]{oeis}.

\section{Yamaguti operad and the adjoint module of \texorpdfstring{$\mathfrak{sl}_2$}{sl2}}\label{sec:adjoint}

Recall that an endomorphism operad of a vector space $V$ is the collection $\mathcal{E}\mathrm{nd}_V$ of multilinear maps from $V$ to $V$:
 \[
\mathcal{E}\mathrm{nd}_V(n)= \mathrm{Hom}(V^{\otimes n},V).     
 \]
In this section, we establish the second main result of this paper (found while this article was under peer review), which relates the Yamaguti operad to the equivariant endomorphism operad 
 \[
(\mathcal{E}\mathrm{nd}_{\mathfrak{sl}_2})^{\mathfrak{sl}_2}=\{\mathrm{Hom}_{\mathfrak{sl}_2}(\mathfrak{sl}_2^{\otimes n},\mathfrak{sl}_2)\}_{n\ge 1} 
 \]
of the adjoint module of the Lie algebra $\mathfrak{sl}_2$. We shall need an invariant symmetric bilinear form $K$ on that adjoint module. It will be convenient for us to choose the form 
 \[
K(u,v):=\frac{1}{2}\mathrm{tr}(\mathrm{ad}_u\mathrm{ad}_v) ,   
 \]
whose nonzero values on the basis $(e,f,h)$ are $K(e,f)=K(f,e)=2$, $K(h,h)=4$.

\begin{proposition}\label{prop:morphsl2}
There is a well-defined morphism of nonsymmetric operads
 \[
\phi\colon\Yam\to (\mathcal{E}\mathrm{nd}_{\mathfrak{sl}_2})^{\mathfrak{sl}_2}   
 \]
defined on the generators by the formula
\begin{align}
  a\cdot b&:= [a,b],\\
  \{ a, b, c \}&:=K(a,b)c,\\
  \lb a, b, c \rb&:=K(b,c)a.
\end{align}
\end{proposition}

\begin{proof}
The map $\phi$ sends the defining identities \eqref{ay1}--\eqref{ay11} of the Yamaguti operad to the following relations in $(\mathcal{E}\mathrm{nd}_{\mathfrak{sl}_2})^{\mathfrak{sl}_2} $:
\begin{gather*}
[[a,b],c]-[a,[b,c]]+K(a,b)c-K(b,c)a=0,\\ 
K([a,b],c)d=K(a,[b,c])d,\\
K(a,b)[c,d]=[K(a,b)c,d],\\
K(c,d)[a,b]=[a,K(c,d)b],\\
K([b,c],d)a=K(b,[c,d])a,\\
[a,K(b,c)d]=[K(b,c)a,d], \\
K(K(a,b)c,d)e = K(a,K(c,d)b)e=K(a,b)K(c,d)e,\\    
K(a,K(b,c)d)e =K(K(b,c)a,d)e,\\
K(d,e)K(b,c)a =K(K(b,c)d,e)a =K(b,K(d,e)c)a,\\
K(K(c,d)b,e)a=K(b,K(c,d)e)a,\\
K(a,b)K(d,e)c=K(d,e)K(a,b)c.
\end{gather*}
The first of them is actually equivalent to the relation  
 \[
[[a,c],b]+K(a,b)c-K(b,c)a=0  
 \]
found by Chmutov and Varchenko \cite[Th.~6]{MR1410469}; in fact, that latter relation is completely classical for vectors in $\mathbb{R}^3$ with the Lie algebra structure given by the cross product. The other relations follow from the bilinearity of the Lie bracket and the form $K$ and the invariance of the form $K$. We conclude that the map $\phi$ is well-defined. 
\end{proof}

\begin{theorem}\label{th:eq-endom}
The morphism $\phi\colon\Yam\to (\mathcal{E}\mathrm{nd}_{\mathfrak{sl}_2})^{\mathfrak{sl}_2}$ is an isomorphism.
\end{theorem}

\begin{proof}
Using the invariant form $K$, we may define an isomorphism of $\mathfrak{sl}_2$-modules $\mathfrak{sl}_2\cong\mathfrak{sl}_2^*$, and hence isomorphisms of $\mathfrak{sl}_2$-modules
 \[
\mathcal{E}\mathrm{nd}_{\mathfrak{sl}_2}(n)=\mathrm{Hom}(\mathfrak{sl}_2^{\otimes n},\mathfrak{sl}_2)\cong \mathfrak{sl}_2^{\otimes(n+1)}.
 \]

Tensor powers of the adjoint representation of $\mathfrak{sl}_2$ are at heart of the notion of the $\mathfrak{sl}_2$ weight system, an object of certain importance in the context of invariants of knots and links; we shall now give a brief recollection of this notion, referring the reader to the monograph \cite{MR2962302} and to the articles \cite{MR1410469,MR3556704} for details. 

Recall that a Jacobi diagram is a graph with univalent and trivalent vertices, with half-edges at each trivalent vertex cyclically ordered, and the univalent vertices carrying (possibly repeating) labels from a finite set $\{1,\ldots,k\}$, modulo the so called AS and IHX relations that are graphical counterparts of the anti-symmetry of Lie brackets and the Jacobi identity in Lie algebras. The vector space $V(k)$ spanned by all such Jacobi diagrams has a commutative associative algebra structure given by the disjoint union. The $\mathfrak{sl}_2$ weight system is an algebra homomorphism from $V(k)$ to the algebra of $\mathfrak{sl}_2$-invariants in $S(\mathfrak{sl}_2)^{\otimes k}$ given on connected Jacobi diagrams (which form a set of free generators of $V(k)$) as follows. One decorates each trivalent vertex of the Jacobi diagram by the invariant $3$-tensor in $\mathfrak{sl}_2^{\otimes 3}$ corresponding to the Lie bracket $[-,-]\in \mathrm{Hom}_{\mathfrak{sl}_2}(\mathfrak{sl}_2^{\otimes 2},\mathfrak{sl}_2)$ under the isomorphism above, and then computes the bilinear form $K$ on pairs of labels of all edges (where we assign a copy of $\mathfrak{sl}_2$ to each of the half-edges adjacent to a vertex following the cyclic ordering).    

Meilhan and Suzuki establish in~\cite[Th.~3.2]{MR3556704} that the vector space $(\mathfrak{sl}_2^{\otimes(n+1)})^{\mathfrak{sl}_2}$ has a basis of images under the $\mathfrak{sl}_2$ weight system of all Jacobi diagrams $D$ whose univalent vertices are bijectively labelled by $\{0,1,\ldots,n+1\}$, whose connected components are ordered linear trees, that is, trees of the form
\[\begin{tikzcd}
    & {i_1} && {i_{k-1}} & \\
    {i_0} & \bullet & \cdots & \bullet & {i_k}
    \arrow[no head, from=2-1, to=2-2]
    \arrow[no head, from=2-2, to=1-2]
    \arrow[no head, from=2-2, to=2-3]
    \arrow[no head, from=2-4, to=1-4]
    \arrow[no head, from=2-4, to=2-3]
    \arrow[no head, from=2-5, to=2-4]
\end{tikzcd}\]
with $i_0<i_1<\cdots<i_{k-1}<i_k$, and for which the partition of $\{0,1,\ldots,n+1\}$ according to the connected components of $D$ is a noncrossing partition without singletons. From our identification
 \[
(\mathcal{E}\mathrm{nd}_{\mathfrak{sl}_2})^{\mathfrak{sl}_2}(n)\cong (\mathfrak{sl}_2^{\otimes(n+1)})^{\mathfrak{sl}_2}     
 \]
and the above description of the $\mathfrak{sl}_2$ weight system, it follows that the elements in $(\mathcal{E}\mathrm{nd}_{\mathfrak{sl}_2})^{\mathfrak{sl}_2}$ corresponding to these elements are obtained by iterations of the composition maps $\circ_i$ from the operations 
\begin{align*}
  a\cdot b&:= [a,b],\\
  \{ a, b, c \}&:=K(a,b)c,\\
  \lb a, b, c \rb&:=K(b,c)a.
\end{align*}
Thus, the morphism $\phi$ of Proposition \ref{prop:morphsl2} is surjective. Moreover, since the above elements form a basis of $(\mathfrak{sl}_2^{\otimes(n+1)})^{\mathfrak{sl}_2}$ and are indexed by noncrossing partitions without singletons, we have $\dim\Yam(n)=\dim (\mathcal{E}\mathrm{nd}_{\mathfrak{sl}_2})^{\mathfrak{sl}_2}(n)$, and so the morphism $\phi$ is an isomorphism.
\end{proof}

Let us remark that collection of tensor powers $\{\mathfrak{sl}_2^{\otimes(n+1)}\}_{n\ge 1}$ has a nonsymmetric operad structure given, on decomposable tensors, by the formula
\begin{multline*}
(v_0\otimes v_1\otimes\cdots\otimes v_n)\circ_i(w_0\otimes w_1\otimes\cdots\otimes w_m)\\ =
K(v_i,w_0)v_0\otimes v_1\otimes\cdots\otimes v_{i-1}\otimes w_1\otimes\cdots \otimes w_m\otimes v_{i+1}\otimes\cdots \otimes v_n, 
\end{multline*}
and our identification 
 \[
\mathcal{E}\mathrm{nd}_{\mathfrak{sl}_2}\cong \{\mathfrak{sl}_2^{\otimes(n+1)}\}_{n\ge 1}
 \]
is an isomorphism of nonsymmetric operads. In fact, if we think of the factors $v_0$, \ldots, $v_n$ as of labels of $n+1$ segments on the boundary of a disk, with the segment $0$ at the bottom, this operad structure is easily seen to be compatible with the cyclic action, which makes $\mathcal{E}\mathrm{nd}_{\mathfrak{sl}_2}$ a nonsymmetric cyclic operad, and $(\mathcal{E}\mathrm{nd}_{\mathfrak{sl}_2})^{\mathfrak{sl}_2}$ its nonsymmetric cyclic suboperad. Examining the proof of Theorem \ref{th:eq-endom}, we see that it establishes an isomorphism with $\Yam$ in the category of nonsymmetric cyclic operads.  

We conclude by indicating a very surprising aspect of Theorem \ref{th:eq-endom}: the components of the equivariant endomorphism operad $(\mathcal{E}\mathrm{nd}_{\mathfrak{sl}_2})^{\mathfrak{sl}_2}$ have natural actions of symmetric groups permuting the arguments of the multilinear maps. Unlike the cyclic symmetry, these actions, while immediate for equivariant maps, are not at all visible on the level of the Yamaguti operad, either for its definition by generators and relations or for its incarnation via noncrossing partitions.

\section{Open questions}

There are several natural questions arising from our work. 

First, we note that for Yamaguti algebras, there is an analogue of the result of Bremner~\cite{MR3265628}: the binary operation of the Yamaguti operad satisfies nontrivial identities. By a computer calculation using Gr\"obner basis of operads \cite{OpGb} (where we assigned weight $0$ to the operation $-\cdot-$ and weight $1$ to the operation $\{-,-,-\}$, and computed the low arities part of the Gr\"obner basis for the corresponding elimination ordering), we established that the lowest degree in which such identity appears is $5$, and the corresponding identity is
 \[
a (b ((c d) e)) + (a ((b c) d)) e  +  ((a b) c) (d e) = a ((b (c d)) e)  +  (a b) (c (d e))  +  ((a (b c)) d) e.   
 \]
It might be interesting to determine all such identities. In the light of the results of Section \ref{sec:adjoint}, it is worth mentioning that we may view the above identity as an identity for the Lie bracket of $\mathfrak{sl}_2$. Using Gr\"obner bases for operads, one can easily check that this identity implies all identities of degree five that hold in $\mathfrak{sl}_2$; it is a well known fact going back to the work of Razmyslov \cite{MR340348} that all identities of the Lie algebra $\mathfrak{sl}_2$ follow from those of degree five, and hence this identity implies all identities of $\mathfrak{sl}_2$. We note that other choices of a single identity of degree five implying all identities of $\mathfrak{sl}_2$ were previously given by Filippov \cite[\S 2]{MR648319} and Mishchenko \cite[Lemma~2.3]{MR1186779}.

Second, the presentation of the Yamaguti operad is quadratic--linear, in that all defining relations are combinations of compositions of at most two generators. It would be interesting to know whether this presentation is inhomogeneous Koszul~\cite{MR2956319,MR1250981}. If that were true, this would give a new approach to the deformation theory of Yamaguti algebras, extending the results of~\cite{das}. The same question may be raised for the Lie--Yamaguti operad, where it is probably much harder.

A weaker form of the previous question is already of independent interest. Let us consider the filtration of the operad $\Yam$ by weight, so that $F^k\Yam(n)$ is the linear span of monomials with $n$ arguments obtained as compositions of at most $k$ generators. Under the isomorphism $\psi$ this corresponds to the filtration of $\mathscr{B}$ by the number of blocks, so that $F^k\mathscr{B}(n)$ is the linear span of noncrossing partitions with at least $n-k$ blocks. In the associated graded operad $\mathrm{gr}_F\mathscr{B}$, the composition is set-theoretical: only the first term in the composition of the operad $\mathscr{B}$ survives in the associated graded. We conjecture that the defining relations of that operad are quadratic (this would be the case if the Yamaguti operad were inhomogeneous Koszul). Note that the dimensions of the weight graded components of $\mathrm{gr}_F\mathscr{B}(n)$ assemble into the local $\gamma$-vector of the cluster subdivision $\Gamma(\Phi)$ associated to the root system $\Phi$ of type $A_{n+1}$, see \cite[Prop.~3.1]{MR2971012}.

The fact that there is a functor of change of operations~\cite{MR3169596} producing from each Yamaguti algebra a Lie--Yamaguti algebra raises a question whether the corresponding morphism of operads has the PBW property~\cite{MR4300233}. Computations with the Poincaré series of the corresponding operads in low arities suggests that it might be possible. 

Finally, it would be extremely interesting to determine what replaces the Yamaguti operad if one replaces $\mathfrak{sl}_2$ by $\mathfrak{sl}_3$ in the results of Section \ref{sec:adjoint}; this would almost certainly offer new interesting insights into the $\mathfrak{sl}_3$ weight system as well as into the identities of the Lie algebra $\mathfrak{sl}_3$, both of which are very far from being fully understood.
\printbibliography{}
\end{document}